\documentclass[a4paper,12pt]{amsart}

\usepackage[T1]{fontenc}
\usepackage[utf8]{inputenc}
\usepackage[english]{babel}
\usepackage{amssymb,amsfonts,amsxtra,amsmath}
\usepackage{stmaryrd}
\usepackage[all,cmtip]{xy}
\usepackage{dsfont}
\renewcommand{\mathbb}{\mathds}
\usepackage{mathrsfs}
\DeclareMathAlphabet{\mathsc}{U}{rsfs}{m}{n}
\usepackage{paralist}
\usepackage[colorlinks]{hyperref}

\newcommand{\nopp}{}

\theoremstyle{plain}
\newtheorem{thm}{Theorem}[section]
 
\newtheorem{prp}[thm]{Proposition} 
\newtheorem{lem}[thm]{Lemma} 
\theoremstyle{definition}

\theoremstyle{remark} 
\newtheorem{rmk}[thm]{Remark}

\newtheorem{exa}[thm]{Example}

\newcommand{\mm}{\mathfrak{m}}
\newcommand{\pp}{\mathfrak{p}}

\renewcommand{\AA}{\mathbb{A}}
\newcommand{\CC}{\mathbb{C}}

\newcommand{\cA}{\mathcal{A}}
\newcommand{\cF}{\mathcal{F}}
\newcommand{\cG}{\mathcal{G}}
\renewcommand{\O}{\mathcal{O}}
\newcommand{\cI}{\mathcal{I}}
\newcommand{\cJ}{\mathcal{J}}
\newcommand{\cM}{\mathcal{M}}
\newcommand{\cS}{\mathcal{S}}

\newcommand{\ideal}[1]{{\left\langle#1\right\rangle}}
\newcommand{\set}[1]{{\left\{#1\right\}}}
\newcommand{\xmid}{\;\middle|\;}

\newcommand{\into}{\hookrightarrow}
\newcommand{\onto}{\twoheadrightarrow}

\newcommand{\p}{\partial}

\DeclareMathOperator{\Ann}{Ann}

\DeclareMathOperator{\ch}{char}
\DeclareMathOperator{\Hom}{Hom}
\DeclareMathOperator{\D}{D}
\DeclareMathOperator{\Der}{Der}

\DeclareMathOperator{\soc}{soc}
\DeclareMathOperator{\Sing}{Sing}
\DeclareMathOperator{\Spec}{Spec}
\DeclareMathOperator{\ST}{ST}
\DeclareMathOperator{\Supp}{Supp}
\DeclareMathOperator{\sAnn}{{\mathcal Ann}}
\DeclareMathOperator{\sDer}{{\mathcal Der}}

\begin{document}

\title{A Saito criterion for holonomic divisors}

\author[R.~Epure]{Raul Epure}
\address{R.~Epure\\
Department of Mathematics\\
TU Kaiserslautern\\
67663 Kaiserslautern\\
Germany}
\email{\href{mailto:epure@mathematik.uni-kl.de}{epure@mathematik.uni-kl.de}}

\author[M.~Schulze]{Mathias Schulze}
\address{M.~Schulze\\
Department of Mathematics\\
TU Kaiserslautern\\
67663 Kaiserslautern\\
Germany}
\email{\href{mailto:mschulze@mathematik.uni-kl.de}{mschulze@mathematik.uni-kl.de}}




\begin{abstract}
We show that a holonomic divisor is free if and only if applying all logarithmic derivations to a generic function with isolated critical point yields a complete intersection Artin algebra.
\end{abstract}

\maketitle

\numberwithin{equation}{section}

\section{Introduction}\label{11}

Abe, Horiguchi, Masuda, Murai and Sato (see \cite{AHMMS16}) consider lower ideals in positive roots of a semisimple complex linear algebraic group.
Associated with such an ideal $I$ is a free ideal arrangement $\cA_I$ of $|I|$ hyperplanes.
Applying all its logarithmic derivations to the quadratic form of the invariant inner product defines a graded complete intersection Artin algebra with socle degree $|I|$.
Terao raised the question about a possible converse of such a construction after a talk of Abe at the 2016 Summer Conference on Hyperplane Arrangements in Sapporo.
The following result gives an answer for holonomic divisors which includes the case of complex hyperplane arrangements.


Consider a reduced hypersurface $D$ in a complex manifold $X$.
Then there is a \emph{logarithmic stratification} $\cS$ of $X$ with respect to $D$.
It is made of smooth connected immersed submanifolds of $X$ satisfying the frontier condition (see \cite[(3.2)]{Sai80}).
For each $p\in S\in\cS$, $\sDer(-\log D)(p)$ generates the tangent space $T_pS$.
If $S$ is locally finite, then $D$ is called a \emph{holonomic divisor} (see \cite[(3.8)]{Sai80}).
Complex hyperplane arrangements are holonomic divisors with logarithmic stratification consisting of relative complements of flats.


\begin{thm}\label{0}
Let $D\subseteq(\CC^n,0)=X$ be a reduced hypersurface singularity defined by $f\in\O_X$ with module of logarithmic derivations $\Der(-\log D)$.

\begin{asparaenum}[(a)]

\item\label{0a} Let $\Gamma\subseteq\mm_X$ be a finite $\CC$-vector space such that $V(\Gamma)=\{0\}$.
If $D$ is holonomic and $\gamma\in\Gamma$ is generic, then $V(\Der(-\log D)(\gamma))=\set{0}$.
In particular, there is an $\O_X$-submodule $\Theta\subseteq\Der(-\log D)$ such that
\[
A_\gamma=\O_X/\Theta(\gamma)
\]
is a complete intersection Artin algebra.

\item\label{0b} Let $\gamma\in\mm_X^2$ with Jacobian ideal $J_\gamma=\Der_\CC(\O_X)(\gamma)$ and $\Theta\subseteq\Der(-\log D)$ an $\O_X$-submodule.
Suppose that $\gamma$ has an isolated critical point and $A_\gamma$ is a complete intersection Artin algebra.
Then $D$ is a free divisor if and only if $f\cdot J_\gamma\subseteq\Theta(\gamma)$, which holds true if $f\cdot\Der_\CC(\O_X)\subseteq\Theta$ or even $\Theta=\Der(-\log D)$. 
In this case the preimages under $\Theta\to\Theta(\gamma)$ of a regular sequence generating $\Theta(\gamma)$ form a basis of $\Der(-\log D)$.

\end{asparaenum}
\end{thm}

Part~\eqref{0b} of Theorem~\ref{0} is proved in \S\ref{12}, part~\eqref{0a} at the end of \S\ref{15}.


\begin{rmk}\label{24}\
\begin{enumerate}[(a)]

\item\label{24a} If $\gamma$ is a generic quadratic form, then $\Ann_{A_\gamma}(\bar J_\gamma)=\soc(A_\gamma)$.

\item\label{24b}
By finite determinacy any $\gamma\in\mm_X^2$ with an isolated critical point is polynomial in terms of suitable coordinates.
If $\Gamma\subseteq\ideal{x_1,\dots,x_n}^2$ is polynomial, then the set of $\gamma\in\Gamma$ with isolated critical point is Zariski open.
We sketch a proof at the end of \S\ref{15}.

\end{enumerate}
\end{rmk}


\begin{rmk}
Abe, Maeno, Murai and Numata established independently a version of Theorem~\ref{0} for central hyperplane arrangements (see \cite[Thm.~1.10]{AMMN18}).
The counterpart $\eta$ of $\gamma$ is a homogeneous form.
The \emph{Solomon--Terao algebra} $\ST(\cA,\eta)$ of a central arrangement $\cA$ plays the role of $A_\gamma$ in case $\Theta=\Der(-\log D)$.
\end{rmk}

\section{Saito's criterion and Wiebe's theorem}\label{12}

Part \eqref{0b} of Theorem~\ref{0} results from a simple combination of Wiebe's theorem and Saito's criterion. 


\begin{proof}[Proof of Theorem~\ref{0}.\eqref{0b}]
Pick coordinates $x_1,\dots,x_n\in\O_X$ and corresponding partial derivatives $\partial_i=\frac{\partial}{\partial x_i}$ for $i=1,\dots,n$.
By hypothesis $g_i=\partial_i(\gamma)$ defines an $\O_X$-sequence $g_1,\dots,g_n\in\mm_X$ and $\Theta(\gamma)=\ideal{f_1,\dots,f_n}$ for some $\O_X$-sequence $f_1,\dots,f_n\in\mm_X$.
For $j=1,\dots,n$, pick $\delta_j\in\Theta$ such that $\delta_j(\gamma)=f_j$.
Any $\delta\in\Der_\CC(\O_X)$ can be written as $\delta=\sum_{i=1}^n\delta(x_i)\partial_i$.
It follows that
\begin{equation}\label{23}
f_j=\delta_j(\gamma)=\sum_{i=1}^n\delta_j(x_i)\partial_i(\gamma)=\sum_{i=1}^n\delta_j(x_i)g_i.
\end{equation}
The Saito determinant $\Delta=\det(\delta_j(x_i))$ is the transition determinant from $g_1,\dots,g_n$ to $f_1,\dots,f_n$.
By Wiebe's theorem (see \cite[E.21]{Kun86})
\[
\ideal{\bar\Delta}=\Ann_{A_\gamma}(\bar J_\gamma)\quad\text{and}\quad
\bar J_\gamma=\Ann_{A_\gamma}(\bar\Delta).
\]
In particular, $\bar\Delta\ne0$ since $J_\gamma\subseteq\mm_X\supseteq\Theta(\gamma)$. 
By reducedness of $D$, $\Delta\in\ideal{f}$ (see \cite[(1.5)]{Sai80}).
It follows that
\[
f\cdot J_\gamma\subseteq\Theta(\gamma)
\iff\bar f\in\Ann_{A_\gamma}(\bar J_\gamma)
\iff\ideal{\Delta}=\ideal{f}
\]
which is in turn equivalent to freeness of $D$ by Saito's criterion (see \cite[(1.8)]{Sai80}).
Moreover $\delta_1,\dots,\delta_n$ is a basis of $\Der(-\log D)$ in this case.
If $f\cdot\Der_\CC(\O_X)\subseteq\Theta$, then $f\cdot J_\gamma\subseteq\Theta(\gamma)$ by applying to $\gamma$.
Note that $f\cdot\Der_\CC(\O_X)\subseteq\Der(-\log D)$ to finish the proof.
\end{proof}


\begin{rmk}
In the setup of Theorem~\ref{0}.\eqref{0b}, suppose that $\gamma$ is homogeneous in terms of the coordinates $x_1,\dots,x_n\in\O_X$.
Applying the Euler identity to $g_1,\dots,g_n$ in \eqref{23} shows that the transition determinant from $x_1,\dots,x_n$ to $f_1,\dots,f_n$ is $\Delta\cdot H_\gamma$ where $H_\gamma$ is the Hessian determinant of $\gamma$.
Then again by Wiebe's theorem
\begin{align*}
\ideal{\Delta}=\ideal{f}
&\iff\ideal{\Delta\cdot H_\gamma}=\ideal{f\cdot H_\gamma}\\
&\iff\overline{f\cdot H_\gamma}\in\Ann_{A_\gamma}(\bar\mm_X)=\soc(A_\gamma).
\end{align*}
\end{rmk}


\begin{exa}
Consider the Whitney umbrella $D\subseteq(\CC^3,0)$ defined by $f=x^2-y^2z$.
A \textsc{Singular} (see \cite{DGPS18}) calculation shows that $\Theta=\Der(-\log D)$ is minimally generated by
\[
\begin{pmatrix}
\delta_1\\
\delta_2\\
\delta_3\\
\delta_4
\end{pmatrix}
=
\begin{pmatrix}
x & 0 & 2z\\
y^2 & 0 & 2x\\
yz & x & 0\\
0 &y & -2z
\end{pmatrix}
\begin{pmatrix}
\p_x\\
\p_y\\
\p_z
\end{pmatrix}
\]
In particular, $D$ is not free.
For generic $\gamma=ax^2+by^2+cz^2$,
\[
\Theta(\gamma)=\ideal{xy,ax^2+2cz^2,by^2-2cz^2,2cxz+axy^2}
\]
and $A_\gamma$ is a non-complete intersection Artin algebra.
Note that 
\[
\bar f=-\frac{2c}a\bar z^2\in\ideal{\bar y\bar z,\bar z^2}=\soc(A_\gamma)
\]
and in particular $f\cdot J_\gamma\subseteq\Theta(\gamma)$.
\end{exa}

\section{Flenner's Bertini theorem}\label{13}

Let $A=(A,\mm_A)$ be a local Noetherian $k$-algebra.
Without explicit mention we shall assume that $\ch(k)=0$ and that $A$ admits a universally finite $k$-derivation $d_{A/k}\colon A\to\D_k(A)$. 
For any factor $k$-algebra $\bar A=A/I$ there is an induced universally finite $k$-derivation $d_{\bar A/k}\colon \bar A\to\D_k(\bar A)=\D_k(A)/\ideal{d_{A/k}(I)}$ (see \cite[\nopp 11.10]{Kun86}).


\begin{rmk}\label{7}
For a surjection $\Gamma\onto\Gamma'$ of finite dimensional $k$-vector spaces genericity holds equivalently for $\gamma\in\Gamma$ or for $\gamma\in\Gamma'$.
Indeed the map identifies with a dominant morphism $\AA^n_k\to\AA^{n'}_k$.
\end{rmk}


The following is a direct application of Flenner's Bertini theorem.


\begin{lem}\label{1}
Let $A$ be a local Noetherian $k$-algebra, $\Gamma\to\mm_A$ a finite dimensional $k$-vector space and $\Theta\subseteq\Der_k(A)$ an $A$-submodule.
Assume that $A$ is regular on $U\subseteq D(\Gamma)$ and that $\Theta$ generates $\Der_k(A)$ on $U$.
Then $U\subseteq D(\Theta(\gamma))$ for generic $\gamma\in\Gamma$.
\end{lem}

\begin{proof}
Let $\pp\in U$.
By hypothesis $A_\pp$ is regular and hence $\D_k(A)_\pp$ is a free $A_\pp$-module (see \cite[(8.7)]{SS72}).
In particular, its dual 
\[
\xymatrix{
\Hom_{A_\pp}(\D_k(A)_\pp,A_\pp)=\Hom_A(\D_k(A),A)_\pp
\ar[r]^-{-\circ d_{A/k}}_-\cong & \Der_k(A)_\pp
}
\]
is free of finite rank.
By hypothesis there are $\delta_1,\dots,\delta_n\in\Theta$ mapping to a basis of $\Der_k(A)_\pp$.
Let $\omega_1,\dots,\omega_n\in\D_k(A)_\pp$ be the dual basis.
By Flenner's local Bertini theorem (see \cite[(4.5), (4.6)]{Fle77} and Remark~\ref{7}) $d_{A/k}(\gamma)$ is basic for $\D_k(A)$ on $D(\Gamma)$ for generic $\gamma\in\Gamma$.
It follows that
\[
\sum_{i=1}^n\delta_i(\gamma)\cdot\omega_i=d_{A/k}(\gamma)\in\D_k(A)_\pp\setminus\pp\D_k(A)_\pp.
\]
Then $\delta_j(\gamma)\not\in\pp^{(1)}=\pp$ for some $j\in\set{1,\dots,n}$ and hence $\pp\in D(\Theta(\gamma))$. 
\end{proof}


We extend the above result to filtrations by closed subschemes.


\begin{prp}\label{2}
Let $A$ be a local Noetherian $k$-algebra, $\Gamma\to\mm_A$ a finite dimensional $k$-vector space and $\Theta\subseteq\Der_k(A)$ an $A$-submodule.
Consider a chain of closed subschemes
\[
\Spec(A)=X_0\supseteq X_1\supseteq\dots\supseteq X_s\supseteq X_{s+1}=V(\Gamma).
\]
Write $X_i=\Spec(A_i)$ where $A_i=A/I_i$. 
For $i=0,\dots,s$ assume that $U_i=X_i\setminus X_{i+1}$ is regular and that $\Theta$ induces $\Theta_i\subseteq\Der_k(A_i)$ which generates $\Der_k(A_i)$ on $U_i$.
Then $V(\Theta(\gamma))\subseteq V(\Gamma)$ for generic $\gamma\in\Gamma$.
\end{prp}

\begin{proof}
Applying Lemma~\ref{1} to $A=A_i$, $U=U_i$, $\Theta=\Theta_i$ yields
\begin{align*}
V(\Theta(\gamma)\mod I_i)&=V(\Theta_i(\gamma\mod I_i))\\
&=X_i\setminus D(\Theta_i(\gamma\mod I_i))
\subseteq X_i\setminus U_i=V(I_{i+1}/I_i).
\end{align*}
It follows that $V(\Theta(\gamma),I_i)\subseteq V(\Theta(\gamma),I_{i+1})$ for $i=0,\dots,s$ and hence 
\[
V(\Theta(\gamma))=V(\Theta(\gamma),I_0)\subseteq\cdots\subseteq V(\Theta(\gamma),I_{s+1})\subseteq V(\Gamma).
\]
for generic $\gamma\in\Gamma$.
\end{proof}

\section{Analytic germs}\label{14}

Assume now that $k$ is algebraically closed and complete non-discretely valued.
Any analytic $k$-algebra admits a universally finite $k$-derivation (see \cite[(2.6)]{SS72}).

Let $X$ be a $k$-analytic space.
For any $x\in X$ analytic subgerms of $(X,x)$ correspond to subschemes of $\Spec(\O_{X,x})$ (see \cite[\S1, Thm.~1.3]{Gro62}).

For two coherent $\O_X$-ideals $\cI$ and $\cJ$, an inclusion of zeros $X\supseteq V(\cI)\supseteq V(\cJ)$ is equivalent to an inclusion of prime ideals $X_x=\Spec(\O_{X,x})\supseteq V(\cI_x)\supseteq V(\cJ_x)$ for all $x\in X$ by the analytic Nullstellensatz (see \cite[\S4,  Cor.~1]{Hou62a}).
Consider $U=X\setminus V(\cI)\subseteq X$ for some coherent $\O_X$-ideal $\cI$ and set $U_x=D(\cI_x)\subseteq X_x$ for any $x\in X$.

Since regularity is an analytic property (see \cite[III.(4.6)]{Sch70}) there is a coherent $\O_X$-ideal $\cJ$ such that $\Sing(X)=V(\cJ)$ and $\Sing(X_x)=V(\cJ_x)$ for all $x\in X$.
Thus $U$ is smooth if and only if $U_x$ is regular for all $x\in X$.

Consider an inclusion of coherent $\O_X$-modules $\cG\subseteq\cF$.
Then $\cG=\cF$ on $U$ is equivalent to $V(\sAnn(\cM))=\Supp(\cM)\subseteq V(\cI)$ where $\cM=\cF/\cG$.
This in turn is equivalent to 
\[
\Supp(\cM_x)=V(\Ann(\cM_x))=V(\sAnn(\cM)_x)\subseteq V(\cI_x)
\]
and hence to $\cG_x=\cF_x$ on $U_x$ for all $x\in X$ (see \cite[B.\S1]{Hou62b}).
For any $\O_{X,x}$-submodule $\Theta\subseteq\Der_k(\O_{X,x})$ with generators $\theta_1,\dots,\theta_m\in\cF(U)$ where $x\in U\subseteq X$ this applies to the coherent $\O_U$-module $\cF=\sDer_k(\O_U)$ with stalk $\sDer_k(\O_U)_x=\Der_k(\O_{X,x})$ and its coherent $\O_U$-submodule $\cG=\ideal{\theta_1,\dots,\theta_m}_{\O_U}$.


Proposition~\ref{2} translates into the following local analytic statement by passing to representatives and using the preceding discussion.


\begin{prp}\label{4}
Let $X$ be the germ of a $k$-analytic space, $\Gamma\to\mm_X$ a finite dimensional $k$-vector space and $\Theta\subseteq\Der_k(\O_X)$ an $\O_X$-submodule.
Consider a chain of closed analytic subgerms
\[
X=X_0\supseteq X_1\supseteq\cdots\supseteq X_s\supseteq X_{s+1}=V(\Gamma).
\]
For $i=0,\dots,s$ assume that $U_i=X_i\setminus X_{i+1}$ is smooth and that $\Theta$ induces $\Theta_i\subseteq\Der_k(\O_{X_i})$ which generates $\sDer_k(\O_{U_i})$.
Then $V(\Theta(\gamma))\subseteq V(\Gamma)$ for generic $\gamma\in\Gamma$.\qed
\end{prp}

\section{Holonomic divisors}\label{15}

Let $D\subseteq X=(\CC^n,0)$ be the germ of a divisor.
Then $X$ has a logarithmic stratification $\cS$ with respect to $D$.
The stratum $S_0\in\cS$ with $0\in S_0$ satisfies $S_0\cong(\CC^d,0)$ where $d=\dim_\CC\sDer(-\log D)(0)$.
Along $S_0$, $D\subseteq X$ is analytically trivial (see \cite[(3.6)]{Sai80}), that is,
\[
(X,D)\cong(X'\times S_0,D'\times S_0)
\]
where $D'\subseteq X'=(\CC^{n-d},0)$ is the germ of a divisor.
There are corresponding decompositions of derivations (see \cite[III.\S5.10]{GR71})
\begin{align*}
\sDer_\CC(\O_X)&=\pi_{X'}^*\sDer_\CC(\O_{X'})\oplus\pi_{S_0}^*\sDer_\CC(\O_{S_0}),\\
\sDer(-\log D)&=\pi_{X'}^*\sDer(-\log D')\oplus\pi_{S_0}^*\sDer_\CC(\O_{S_0}),
\end{align*}
where $\pi_{X'}$ and $\pi_{S_0}$ denote the canonical projections.
Denote by $\cS'$ the logarithmic stratification of $X'$ with respect to $D'$.
For $p=(p',p_0)\in S\in\cS$ and $p'\in S'\in\cS'$,
\begin{align*}
T_pS&=\sDer(-\log D)(p)\\
&=\sDer(-\log D')(p')\oplus\sDer_\CC(\O_{S_0})(p_0)
=T_{p'}S'\oplus T_{p_0}S_0.
\end{align*}
This yields a bijection of logarithmic strata
\[
\cS'\leftrightarrow\cS,\quad S'\mapsto S'\times S_0,\quad S\cap X'\mapsfrom S.
\]
In particular, $S_0'=S_0\cap X'=\set{0}\in\cS'$ plays the role of $S_0$ for $D'\subseteq X'$.
Then $\sDer(-\log D')(0)=T_0S_0'=0$ and hence $\sDer(-\log D)$ preserves the ideal of $S_0$.
The resulting canonical map
\[
\sDer(-\log D)\to\pi_{S_0}^*\sDer_\CC(\O_{S_0})
\]
induces a surjection
\begin{align*}
\sDer(-\log D)_0\onto\sDer(-\log D)(0)&=T_0S_0\\
&=\sDer_\CC(\O_{S_0})_0/\mm_{S_0,0}\sDer_\CC(\O_{S_0})_0.
\end{align*}
By Nakayama's lemma it follows that $\sDer(-\log D)_0$ induces generators of  $\sDer_\CC(\O_{S_0})_0$.

Assume now that $D$ (or equivalently $D'$) is holonomic.
Then $\cS$ and $\cS'$ are finite, $S_0'\in\cS'$ is the unique point stratum and hence $S_0\in\cS$ is the unique stratum with minimal dimension $d$.
Consider the closed analytic subgerms (see \cite[(3.12)]{Sai80})
\[
X_i=\set{p\in X\xmid\dim_\CC\sDer(-\log D)(p)\le n-i}\subseteq X,
\]
for $i=0,\dots,s=n-d$.
Note that $X_0\setminus X_1=X\setminus D$ is a stratum, $X_1\setminus X_2=D\setminus\Sing(D)$, and $X_s=S_0$.
More generally (see \cite[(3.13)]{Sai80})
\[
U_i=X_i\setminus X_{i+1}=\bigcup_{\substack{S\in\cS\\\dim S=n-i}}S
\]
for $i=0,\dots,s-1$.
For each $i\in\set{0,\dots,s-1}$ and $p\in U_i$, uniqueness of $S_p$ yields
\[
(U_i,p)=(S_p,p).
\]
By the preceding discussion applied at $p$, $\sDer(-\log D)_p$ induces generators of $\sDer_\CC(\O_{U_i})_p$.
With coherence of $\sDer(-\log D)$ it follows that 
\[
\Theta=\Der(-\log D)=\sDer(-\log D)_0
\]
induces generators of $\sDer_\CC(\O_{U_i})$.


Part~\eqref{0a} of Theorem~\ref{0} now reduces to Proposition~\ref{4}.


\begin{proof}[Proof of Theorem~\ref{0}.\eqref{0a}]
Due to the preceding discussion Proposition~\ref{4} applies with $\Theta=\Der(-\log D)$.
By choice of $\Gamma$ it yields
\[
V(\Der(-\log D)(\gamma))\subseteq V(\Gamma)=\set{0}.
\]
Then $\Theta$ is generated by preimages under $\Der(-\log D)\to\Der(-\log D)(\gamma)$ of a regular sequence (see \cite[Cor.~2.2.6, Thm.~2.1.2.(b)]{BH93}).
\end{proof}


\begin{proof}[Proof of Remark~\ref{24}.\eqref{24b}]
Set $x=x_1,\dots,x_n$ and $\partial_i=\frac{\partial}{\partial x_i}$ for $i=1,\dots,n$.
The linear map $\Gamma\to\partial\Gamma=\bigoplus_{i=1}^n\partial_i\Gamma$ sends $\gamma$ with isolated critical point to a regular sequence $\partial\gamma=\partial_1\gamma,\dots,\partial_n\gamma\in\CC[x]_{\ideal{x}}$.
The ideal $\ideal{\omega_1,\dots,\omega_n}$ generated by the inclusions
\[
(\omega_i\colon\partial_i\Gamma\into\CC[x])\in\CC[x]\otimes(\partial_i\Gamma)^\vee\subseteq\CC[x]\otimes\CC[\partial\Gamma]
\]
defines a subscheme $W\subseteq\AA_\CC^n\times\partial\Gamma$, containing $\set{0}\times\partial\Gamma$ by hypothesis.
Since $\partial\gamma$ is a regular sequence, the fiber of $\varphi\colon W\to\partial\Gamma$ over $\partial\gamma$ is a zero-dimensional complete intersection at $(0,\partial\gamma)$.
By Chevalley's theorem (see \cite[Thm.~(13.1.3)]{EGA4}) the same holds at all points in an affine neighborhood $U$ of $(0,\partial\gamma)$.
Since $\partial\Gamma$ is regular it follows that $\omega_1,\dots,\omega_n$ is a regular sequence at all points in $U$.
In other words $\varphi\colon U\to\partial\Gamma$ is a relative complete intersection (see \cite[Def.~(19.3.6)]{EGA4}).
In particular, it is flat and open (see \cite[Thm.~(2.4.6)]{EGA4}).
It follows that $\varphi(U)$ is an open neighborhood of $\partial\gamma$ containing only regular sequences in $\CC[x]_{\ideal{x}}$.
This makes $\Gamma\cap\varphi(U)$ an open neighborhood of $\gamma$ consisting of polynomials with isolated critical point.
\end{proof}


\begin{exa}\label{17}
Let $D\subseteq X=(\CC^3,0)$ be defined by $f=xy(x+y)(x-y)(y-xz)$ and $\Gamma=\ideal{x^2,y^2,z^2}_\CC$. 
Using the Jacobian criterion, one computes that 
\[
\Sing(D)=V(y,z)\cup V(x,y)=V(y,xz)\supseteq V(\Gamma)=\set{0}
\]
is the union of the $x$-axis and the $z$-axis. 
So $U=X\setminus V(y,xz)\subseteq X\setminus V(\Gamma)$ is smooth. 
A \textsc{Singular} (see \cite{DGPS18}) calculation shows that $\Theta=\Der(-\log D)$ is generated by
\[
\begin{pmatrix}
\delta_1\\
\delta_2\\
\delta_3
\end{pmatrix}
=
\begin{pmatrix}
x & y & 0\\
0 & x^2y-y^3 & 2xy-x^2z-3y^2z+2xyz^2\\
0 & 0 & y-xz
\end{pmatrix}
\begin{pmatrix}
\p_x\\
\p_y\\
\p_z
\end{pmatrix}
\]
making $D$ a free divisor.
On the punctured $x$-axis, $\delta_1=x\p_x$ is non-vanishing.
The $z$-axis however is a union of non-holonomic point strata since all $\delta_i$ vanish on it.
It follows that the $z$-axis is the non-holonomic locus of $D$.
For generic $\gamma=ax^2+by^2+cz^2\in\Gamma$,
\begin{align*}
\Theta(\gamma)=\langle
&2ax^2+2by^2,\\
&4cxyz+2bx^2y^2-2by^4-2cx^2z^2-6cy^2z^2+4cxyz^3,\\
&2cyz-2cxz^2
\rangle
\end{align*}
and $V(\Theta(\gamma))=V(x,y)\supsetneq V(\Gamma)$ is the non-holonomic locus of $D$.
\end{exa}

\bibliographystyle{amsalpha}
\bibliography{berder-arxiv}
\end{document}